\documentclass[12pt]{article}
\usepackage{amsmath,amsthm,amssymb}
\usepackage{amssymb,latexsym}
\newtheorem{theorem}{Theorem}

\newtheorem{lemma}{Lemma}

\begin{document}

\baselineskip=17pt

\title{\bf A binary additive equation with prime and square-free number}

\author{\bf S. I. Dimitrov}

\date{}

\maketitle
\begin{abstract}
Let $[\, \cdot\,]$ be the floor function.
In this paper, we show that when $1<c<\frac{82}{79}$, then every sufficiently large positive integer $N$ can be represented in the form   
\begin{equation*}
N=[p^c]+[m^c]\,,
\end{equation*}
where $p$ is a prime and $m$ is a square-free.
\quad\\
{\bf  2020 Math.\ Subject Classification}: 11L07 $\cdot$ 11L20 $\cdot$ 11P05 $\cdot$ 11P32
\end{abstract}

\section{Introduction and statement of the result}
\indent

In 1974 Deshouillers \cite{Deshouillers} proved that if $1 <c<\frac{4}{3}$, then for every sufficiently large    
positive integer $N$ the binary equation
\begin{equation}\label{binary1}
[m^c_1]+[m^c_2]=N\,,
\end{equation}
has a solution with $m_1$ and $m_2$ positive integers.
Subsequently, the range for $c$ in this result was extended by Gritsenko \cite{Gritsenko} to $1 <c<\frac{55}{41}$
and by Konyagin \cite{Konyagin} to $1 <c<\frac{3}{2}$.

Considered with primes, the equation \eqref{binary1} is considerably more difficult, probably at least as hard as the binary Goldbach problem.
The first step in this direction was made by Laporta \cite{Laporta}, which showed that if $1<c<\frac{17}{16}$,
then for almost all $N$  the equation \eqref{binary1} has a solution with primes $m_1$, $m_2$.
Afterwards W. Zhu \cite{WZhu}, L. Zhu \cite{LZhu} and Baker \cite{Baker} improved the result of Laporta in various ways.
In 2007 Balanzario, Garaev and Zuazua \cite{Balanzario} proved that when $1 <c<\frac{17}{11}$,
then for almost all positive integers $N$ the binary equation
\begin{equation}\label{binary2}
[p^c]+[m^c]=N\,,
\end{equation}
has a solution with prime $p$ and positive integer $m$.
In turn Kumchev \cite{Kumchev} showed that if $1<c<\frac{16}{15}$, then the equation \eqref{binary2} is solvable
for sufficiently large positive integer $N$. The result of Kumchev was sharpened by Yu \cite{Yu} with $1<c<\frac{11}{10}$.
In continuation of these studies Wu \cite{Wu} proved the solvability of the equation \eqref{binary2}
for $1 <c<\frac{247}{238}$ with an almost prime $m$ with at most $\left[\frac{450}{247-238c}\right]+1$ prime factors.
A weaker result was previously obtained by Petrov and Tolev \cite{Petrov-Tolev}.
On the other hand many diophantine equations are solved with square-free numbers.
For example it was first shown by Estermann \cite{Estermann} in 1931 that every sufﬁciently large
positive integer can be represented as the sum of a prime and a square-free number.
Subsequently, the result of Estermann was improved by Page \cite{Page} and Mirsky \cite{Mirsky}.
Finally Dudek \cite{Dudek} proved that every integer greater than two may be written as the
sum of a prime and a square-free number. In turn Li \cite{Li} sharpened the theorem of Dudek.
But many other similar articles can be found in literature.
Recently the author \cite{Dimitrov} showed that when $1<c<\frac{3849}{3334}$, then there exist infinitely many prime numbers of the form $[n^c]$, where $n$ is a square-free.
Motivated by all the mentioned results, we establish the following theorem.
\begin{theorem}\label{Theorem} Suppose that $1<c<\frac{82}{79}$.
Then for every sufficiently large positive integer $N$ the binary equation
\begin{equation*}
[p^c]+[m^c]=N
\end{equation*}
has a solution in prime number $p$ and square-free number $m$.
\end{theorem}

\section{Notations}
\indent

Let $N$ be a sufficiently large positive integer.
The letter $p$ will always denote prime number.
By $\varepsilon$ we denote an arbitrary small positive number, not the same in all appearances.
As usual $\mu (n)$, $\Lambda(n)$ and $\tau(n)$ denote respectively M\"{o}bius' function,
von Mangoldt's function and the number of positive divisors of $n$.
We shall use the convention that a congruence, $m\equiv n\,\pmod {d}$ will be written as $m\equiv n\,(d)$.
We denote by $[t]$ and $\{t\}$ the integer part of $t$ and the fractional part of $t$, respectively.
Moreover $\psi(t)=\{t\}-1/2$ and $e(x)=e^{2\pi i x}$.
Throughout this paper we suppose that $1<c<\frac{82}{79}$ and $\gamma=\frac{1}{c}$.

Denote
\begin{align}\label{P}
&P=10^{-9}N^\gamma\,;\\
\label{GammaP}
&\Gamma=\sum\limits_{P<p\leq 2P, m\in \mathbb{N}\atop{[p^c]+[m^c]=N}}\mu^2(m)\log p\,;\\
\label{z}
&D=\log N\,.
\end{align}

\section{Preliminary lemmas}
\indent

\begin{lemma}\label{Vaaler}
For every $H\geq1$, we have
\begin{equation*}
\psi(t)=\sum\limits_{1\leq|h|\leq H}a(h)e(ht)+\mathcal{O}\Bigg(\sum\limits_{|h|\leq H}b(h)e(ht)\Bigg)\,,
\end{equation*}
where
\begin{equation}\label{ahbh}
a(h)\ll\frac{1}{|h|}\,,\quad b(h)\ll\frac{1}{H}\,.
\end{equation}
\end{lemma}
\begin{proof}
See \cite{Vaaler}.
\end{proof}

%
\begin{lemma}\label{PartitionofUnity1} Let $Z\ge 2$ be an integer and let $0\leq z\leq2Z-1$ be fixed.
Then there exists a periodic function $\theta_z(x)$ with period 1 such that
\begin{align*}
&(\mathrm{i}) \quad 0<\theta_z(x)<1 \quad \mbox{ if } \quad \left|x-\frac{z}{2Z}\right|<\frac{1}{2Z}\,;\\
&(\mathrm{ii}) \quad \theta_z(x)=0 \quad \mbox{ if } \quad \frac{1}{2Z}\leq\left|x-\frac{z}{2Z}\right|\leq\frac{1}{2}\,;\\
&(\mathrm{iii}) \;\; \mbox{ The Fourier expansion of } \theta_z(x) \mbox{ is of the form }
\end{align*}
\begin{equation}\label{Fourierexp}
\theta_z(x)=\sum\limits_{|n|\leq Z(\log N)^4}g_z(n)e(nx)+\mathcal{O}\Big(N^{-\log\log N}\Big)\,,
\end{equation}
where
\begin{equation}\label{gznZ}
\big|g_z(n)\big|\le\frac{1}{2Z}\,.
\end{equation}
We also have
\begin{equation}\label{PartitionofUnity2}
\sum\limits_{z=0}^{2Z-1}\theta_z(x)=1\,.
\end{equation}
\end{lemma}
\begin{proof}
See (\cite{Karatsuba}, Chapter 1, Lemma A and \cite{Petrov-Tolev}, Section 3.3).
\end{proof}

\begin{lemma}\label{Vaughanidentity}
Let $u, N, N_1\in\mathbb{R}$, $1<u\leq N<N_1$ and let $f(n)$ be an arbitrary function defined for $n\in\mathbb{N}$, $n\in(N, N_1]$. Then
\begin{equation*}
\sum_{N<n\le N_1}\Lambda(n)f(n)=S_1-S_2-S_3\,,
\end{equation*}
where

\begin{align*}
&S_1=\sum_{m\le u}\mu(m)\sum_{\frac{N}{m}<l\le\frac{N_1}{m}}(\log l)f(ml)\,,\\
&S_2=\sum_{m\le u^2}c(m)\sum_{\frac{N}{m}<l\le\frac{N_1}{m}}f(ml)\,,\\
&S_3=\mathop{\sum\sum}_{\substack{N<ml\le N_1 \\m>u,\,l>u}}a(m)\Lambda(l)f(ml)\,,
\end{align*}
where
\begin{equation*}
|a(m)|\leq\tau(m)\,,\quad |c(m)|\leq\log m\,.
\end{equation*}
\end{lemma}
\begin{proof}
See (\cite{Vaughan}).
\end{proof}

\begin{lemma}\label{Iwaniec-Kowalski}
For any complex numbers $a(n)$ we have
\begin{equation*}
\bigg|\sum_{a<n\le b}a(n)\bigg|^2
\leq\bigg(1+\frac{b-a}{Q}\bigg)\sum_{|q|\leq Q}\bigg(1-\frac{|q|}{Q}\bigg)\sum_{a<n,\, n+q\leq b}a(n+q)\overline{a(n)}\,,
\end{equation*}
where $Q$ is any positive integer.
\end{lemma}
\begin{proof}
See (\cite{Iwaniec-Kowalski}, Lemma 8.17).
\end{proof}

\begin{lemma}\label{Exponentpairs}
Let $|f^{(k)}(u)|\asymp YX^{1-k}$  for $1\leq X<u<X_0\leq2X$ and $k\geq1$.\\
Then
\begin{equation*}
\bigg|\sum_{X<n\le X_0}e\big(f(n)\big)\bigg|\ll Y^\varkappa X^\lambda +Y^{-1},
\end{equation*}
where $(\varkappa, \lambda)$ is any exponent pair.
\end{lemma}
\begin{proof}
See (\cite{Graham-Kolesnik}, Ch. 3).
\end{proof}

\begin{lemma}\label{ExppairBourgain}
For every $\varepsilon > 0$, the pair $\Big(\frac{13}{84}+\varepsilon, \frac{55}{84}+\varepsilon\Big)$ is an exponent pair.
\end{lemma}
\begin{proof}
See (\cite{Bourgain}, Theorem 6).
\end{proof}

\section{Beginning of the proof}
\indent

We consider the sum  $\Gamma$ defined by \eqref{GammaP}.
The theorem will be proved if we show that $\Gamma>0$.
Our first maneuvers are straightforward.
Using \eqref{GammaP}, \eqref{z} and the well-known identity
\begin{equation*}
\mu^2(n)=\sum_{d^2|n}\mu(d)
\end{equation*}
we write
\begin{align}\label{Gammadecopm}
\Gamma&=\sum\limits_{P<p\leq 2P, m\in \mathbb{N}\atop{[p^c]+[m^c]=N}}(\log p)\sum_{d^2|m}\mu(d)
=\Gamma_1+\Gamma_2\,,
\end{align}
where
\begin{align}
\label{Gamma1}
&\Gamma_1=\sum_{d\leq D}\mu(d)\sum_{P<p\leq 2P}\log p
\sum\limits_{m\in\mathbb{N}\atop{m\equiv 0\,(d^2)\atop{[p^c]+[m^c]=N}}}1\,,\\
\label{Gamma2}
&\Gamma_2=\sum_{d>D}\mu(d)\sum_{P<p\leq 2P}\log p
\sum\limits_{m\in\mathbb{N}\atop{m\equiv 0\,(d^2)\atop{[p^c]+[m^c]=N}}}1\,.
\end{align}
We shall estimate $\Gamma_1$ and $\Gamma_2$, respectively, in the sections \ref{SectionGamma1} and \ref{SectionGamma2}.
In section \ref{Sectionfinal} we shall finalize the proof of Theorem \ref{Theorem}.

\section{Estimation of $\mathbf{\Gamma_1}$}\label{SectionGamma1}
\indent

Our argument is a modification of Petrov-Tolev's \cite{Petrov-Tolev} and Wu's \cite{Wu} argument.
Using the identity 
\begin{equation*}
\sum_{a\leq m<b}1=[-a]-[-b]=b-a-\psi(-a)+\psi(-b)
\end{equation*}
we obtain
\begin{align}\label{innersum}
\sum\limits_{m\in\mathbb{N}\atop{m\equiv 0\,(d^2)\atop{[p^c]+[m^c]=N}}}1
&=\sum\limits_{m\in\mathbb{N}\atop{m\equiv 0\,(d^2)\atop{N-[p^c]\leq m^c<N+1-[p^c]}}}1
=\sum\limits_{\frac{1}{d^2}(N-[p^c])^\gamma\leq m<\frac{1}{d^2}(N+1-[p^c])^\gamma}1\nonumber\\
&=\frac{\big(N+1-[p^c]\big)^\gamma-\big(N-[p^c]\big)^\gamma}{d^2}-\psi\left(-\frac{1}{d^2}\big(N-[p^c]\big)^\gamma\right)\nonumber\\
&+\psi\left(-\frac{1}{d^2}\big(N+1-[p^c]\big)^\gamma\right)\,.
\end{align}
Now \eqref{Gamma1} and \eqref{innersum} give us
\begin{align}\label{Gamma1est1}
\Gamma_1&=\sum_{d\leq D}\frac{\mu(d)}{d^2}\sum_{P<p\leq 2P}\Big(\big(N+1-[p^c]\big)^\gamma-\big(N-[p^c]\big)^\gamma\Big)\log p\nonumber\\
&+\sum_{d\leq D}\mu(d)\sum_{P<p\leq 2P}\Bigg(\psi\left(-\frac{1}{d^2}\big(N+1-[p^c]\big)^\gamma\right)
-\psi\left(-\frac{1}{d^2}\big(N-[p^c]\big)^\gamma\right)\Bigg)\log p\nonumber\\
&=\Gamma_3-\Sigma_0+\Sigma_1\,,
\end{align}
where
\begin{align}
\label{Gamma3}
&\Gamma_3=\sum_{d\leq D}\frac{\mu(d)}{d^2}\sum_{P<p\leq 2P}\Big(\big(N+1-[p^c]\big)^\gamma-\big(N-[p^c]\big)^\gamma\Big)\log p\,,\\
\label{Sigmaj}
&\Sigma_j=\sum_{d\leq D}\mu(d)\sum_{P<p\leq 2P}\psi\left(-\frac{1}{d^2}\big(N+j-[p^c]\big)^\gamma\right)\log p\,, \quad j=0, 1\,.
\end{align}

\subsection{Lower bound for $\mathbf{\Gamma_3}$}
\indent

On the one hand 
\begin{equation}\label{Binomialseries}
\big(N+1-[p^c]\big)^\gamma-\big(N-[p^c]\big)^\gamma=\gamma\big(N-[p^c]\big)^{\gamma-1}+\mathcal{O}\left(N^{\gamma-2}\right)\,.
\end{equation}
On the other hand 
\begin{equation}\label{Gegenbauer}
\sum_{d\leq D}\frac{\mu(d)}{d^2}=\frac{6}{\pi^2}+\mathcal{O}\big(D^{-1}\big)\,.
\end{equation}
Now \eqref{P}, \eqref{Gamma3}, \eqref{Binomialseries}, \eqref{Gegenbauer} and Chebyshev's prime number theorem yield
\begin{equation}\label{Gamma3est}
\Gamma_3\gg N^{2\gamma-1}\,.
\end{equation}

\subsection{Estimation of $\mathbf{\Sigma_0}$ and $\mathbf{\Sigma_1}$}
\indent

Using \eqref{Sigmaj} and Lemma \ref{Vaaler} we get
\begin{equation}\label{Sigmajdecomp}
\Sigma_j=\Sigma'_j+\mathcal{O}\big(\Sigma''_j\big)\,,
\end{equation}
where
\begin{align}
\label{Sigma'j}
&\Sigma'_j=\sum_{d\leq D}\mu(d)\sum_{P<p\leq 2P}(\log p)\sum\limits_{1\leq|h|\leq H}a(h)e\left(-\frac{h}{d^2}\big(N+j-[p^c]\big)^\gamma\right)\,,\\
\label{Sigma''j}
&\Sigma''_j=\sum_{d\leq D}\mu(d)\sum_{P<p\leq 2P}(\log p)\sum\limits_{|h|\leq H}b(h)e\left(-\frac{h}{d^2}\big(N+j-[p^c]\big)^\gamma\right)\,.
\end{align}
Put
\begin{equation}\label{Wv}
W(v)=\sum_{P<p\leq 2P}(\log p)e\left(v\big(N+j-[p^c]\big)^\gamma\right)\,.
\end{equation}
First we consider the sum $\Sigma'_j$. From \eqref{ahbh}, \eqref{Sigma'j}, \eqref{Wv} and changing the order of summation, we deduce
\begin{equation}\label{Sigma'jest}
\Sigma'_j=\sum_{d\leq D}\mu(d)\sum\limits_{1\leq|h|\leq H}a(h)W\left(-\frac{h}{d^2}\right)
\ll\sum_{d\leq D}\sum\limits_{1\leq h\leq H}\frac{1}{h}\left|W\left(\frac{h}{d^2}\right)\right|\,.
\end{equation}
Next we consider the sum $\Sigma''_j$. By \eqref{ahbh}, \eqref{Sigma''j} and \eqref{Wv}, we derive
\begin{align}\label{Sigma''jest}
\Sigma''_j&\ll\sum_{d\leq D}\sum_{P<p\leq 2P}(\log p)\sum\limits_{|h|\leq H}b(h)e\left(-\frac{h}{d^2}\big(N+j-[p^c]\big)^\gamma\right)\nonumber\\
&=\sum_{d\leq D}\sum\limits_{|h|\leq H}b(h)W\left(-\frac{h}{d^2}\right)
\ll\sum_{d\leq D}\sum\limits_{|h|\leq H}\frac{1}{H}\left|W\left(\frac{h}{d^2}\right)\right|\nonumber\\
&\ll\sum_{d\leq D}\frac{1}{H}\big|W(0)\big|+\sum_{d\leq D}\sum\limits_{1\leq h\leq H}\frac{1}{H}\left|W\left(\frac{h}{d^2}\right)\right|\,.
\end{align}
Bearing in mind \eqref{P}, \eqref{Wv} and Chebyshev's prime number theorem, we obtain
\begin{equation}\label{W0est}
W(0)\asymp N^ \gamma\,.
\end{equation}
Set
\begin{equation}\label{HdN}
H=d^2N^{1-\gamma}\log N\,.
\end{equation}
Now \eqref{z}, \eqref{Sigmajdecomp}, \eqref{Sigma'jest} -- \eqref{HdN} imply
\begin{equation}\label{Sigmajest1}
\Sigma_j\ll \frac{N^{2\gamma-1}}{\log N}+\sum_{d\leq D}\sum\limits_{h\leq H}\frac{1}{h}\left|W\left(\frac{h}{d^2}\right)\right|\,, \quad j=0, 1\,.
\end{equation}

\subsubsection{Consideration of the sum $\mathbf{\textit{W(v)}}$}
\indent

Hence forth we assume that
\begin{equation}\label{vdh}
v=\frac{h}{d^2}\,, \quad \mbox{where} \quad 1\leq d \leq D\,, \quad 1\leq h \leq H\,. 
\end{equation}
Using \eqref{Wv} and Lemma \ref{PartitionofUnity1} with
\begin{equation}\label{ZdN}
Z\asymp d^2 N^{1-\gamma}\log^3N
\end{equation}
we write
\begin{equation}\label{Wvest1}
W(v)=\sum_{P<p\leq 2P}(\log p)e\left(v\big(N+j-[p^c]\big)^\gamma\right)\sum\limits_{z=0}^{2Z-1}\theta_z(p^c)=\sum\limits_{z=0}^{2Z-1}W_z(v)\,,
\end{equation}
where
\begin{equation}\label{Wzv}
W_z(v)=\sum_{P<p\leq 2P}(\log p)\theta_z(p^c)e\left(v\big(N+j-[p^c]\big)^\gamma\right)\,.
\end{equation}
Taking into account \eqref{P}, \eqref{ZdN}, \eqref{Wzv} and arguing as in \cite{Petrov-Tolev}, we find
\begin{equation}\label{W0vest}
W_0(v)\ll\frac{N^ {2\gamma-1}}{d^2\log^2N}
\end{equation}
and for $z\neq0$
\begin{equation}\label{Wzvest1}
W_z(v)=V_z(v)+\mathcal{O}\Bigg(\frac{vN^ {2\gamma-2}}{d^2\log^3N}\sum_{P<p\leq 2P}(\log p)\theta_z(p^c)\Bigg)\,,
\end{equation}
where
\begin{equation}\label{Vzv}
V_z(v)=\sum_{P<p\leq 2P}(\log p)\theta_z(p^c)e\Bigg(v\bigg(N+j-p^c+\frac{z}{2Z}\bigg)^\gamma\Bigg)\,.
\end{equation}
From \eqref{Wvest1}, \eqref{W0vest} and \eqref{Wzvest1}, we get
\begin{equation}\label{Wvest2}
W(v)=\sum\limits_{z=1}^{2Z-1}W_z(v)+W_0(v)=\sum\limits_{z=1}^{2Z-1}V_z(v)+\mathcal{O}\big(\Xi\big)+\mathcal{O}\Bigg(\frac{N^ {2\gamma-1}}{d^2\log^2N}\Bigg)\,,
\end{equation}
where
\begin{equation}\label{Xi}
\Xi=\frac{vN^ {2\gamma-2}}{d^2\log^3N}\sum_{P<p\leq 2P}(\log p)\sum\limits_{z=1}^{2Z-1}\theta_z(p^c)\,.
\end{equation}
Now \eqref{P}, \eqref{PartitionofUnity2}, \eqref{Xi} and Chebyshev's prime number theorem give us
\begin{equation*}
\Xi\ll\frac{vN^ {3\gamma-2}}{d^2\log^3N}
\end{equation*}
which together with \eqref{Wvest2}  yields
\begin{equation}\label{Wvest3}
W(v)=\sum\limits_{z=1}^{2Z-1}V_z(v)+\mathcal{O}\Bigg(\frac{vN^ {3\gamma-2}}{d^2\log^3N}\Bigg)+\mathcal{O}\Bigg(\frac{N^ {2\gamma-1}}{d^2\log^2N}\Bigg)\,.
\end{equation}
Furthermore \eqref{HdN}, \eqref{vdh} and \eqref{Wvest3} imply
\begin{equation}\label{Wvest4}
W(v)=\sum\limits_{z=1}^{2Z-1}V_z(v)+\mathcal{O}\Bigg(\frac{N^ {2\gamma-1}}{d^2\log^2N}\Bigg)\,.
\end{equation}
By \eqref{Fourierexp}, \eqref{gznZ}, \eqref{ZdN} and \eqref{Vzv}, we obtain
\begin{equation}\label{Vzvest}
V_z(v)\ll N^{-10}+\frac{1}{Z}\sum_{|r|\leq R}\sup_{T\in[N, N+2]}\big|U(T, r, v)\big|\,,
\end{equation}
where
\begin{equation}\label{UTrv}
U(T, r, v)=\sum_{P<p\leq 2P}(\log p)e\big(rp^c+v(T-p^c)^\gamma\big)\,,
\end{equation}
\begin{equation}\label{RdN}
R=d^2N^{1-\gamma}\log^8N\,.
\end{equation}

\subsubsection{Application of Vaughan's identity}
\indent

Combining \eqref{Sigmajest1}, \eqref{ZdN}, \eqref{Wvest4} and \eqref{Vzvest} we get
\begin{equation}\label{Sigmajest2}
|\Sigma_j|\ll \frac{N^{2\gamma-1}}{\log N}+\sum_{d\leq D}\sum\limits_{h\leq H}\frac{1}{h}\sum_{|r|\leq R}\sup_{T\in[N, N+2]}\big|U(T, r, v)\big|\,, \quad j=0, 1\,.
\end{equation}
Denote
\begin{equation}\label{htT}
h(t)=rt^c+v(T-t^c)^\gamma\,,
\end{equation}
\begin{equation}\label{fml}
f(m, l)=h(ml)=r(ml)^c+v\big(T-(ml)^c\big)^\gamma\,.
\end{equation}
It is clear that \eqref{UTrv} and \eqref{htT} lead to
\begin{equation}\label{UTrvLambdan}
U(T, r, v)=\sum_{P<n\leq 2P}\Lambda(n)e\big(h(n)\big)+\mathcal{O}\big(P^{1/2}\big)\,.
\end{equation}
Using \eqref{fml}, \eqref{UTrvLambdan} and Lemma \ref{Vaughanidentity} with parameters $u=P^{\frac{1}{3}}$, $N=P$, $N_1=2P$, we derive  
\begin{equation}\label{Ualphadecomp}
U(T, r, v)=U_1-U_2-U_3-U_4+\mathcal{O}\big(P^{1/2}\big)\,,
\end{equation}
where
\begin{align}
\label{U1}
&U_1=\sum_{m\le P^{1/3}}\mu(m)\sum_{P/m<l\le 2P/m}(\log l)e\big(f(m,l)\big)\,,\\
\label{U2}
&U_2=\sum_{m\le P^{1/3}}c(m)\sum_{P/m<l\le 2P/m}e\big(f(m,l)\big)\,,\\
\label{U3}
&U_3=\sum_{P^{1/3}<m\le P^{2/3}}c(m)\sum_{P/m<l\le 2P/m}e\big(f(m,l)\big)\,,\\
\label{U4}
&U_4= \mathop{\sum\sum}_{\substack{P<ml\le 2P \\m>P^{1/3},\,l>P^{1/3} }}a(m)\Lambda(l)e\big(f(m,l)\big)\,,
\end{align}
and where
\begin{equation*}
|c(m)|\leq\log m\,, \quad  |a(m)|\leq\tau(m)\,.
\end{equation*}
Bearing in mind \eqref{P}, \eqref{z}, \eqref{HdN}, \eqref{RdN}, \eqref{Sigmajest2} and \eqref{Ualphadecomp}, we deduce
\begin{equation}\label{Sigmajest3}
|\Sigma_j|\ll \frac{N^{2\gamma-1}}{\log N}+\sum_{i=1}^4\Omega_i\,, \quad j=0, 1\,,
\end{equation}
where
\begin{equation}\label{Omegai}
\Omega_i=\sum_{d\leq D}\sum\limits_{h\leq H}\frac{1}{h}\sum_{|r|\leq R}\sup_{T\in[N, N+2]}|U_i|\,, \quad i=1, 2, 3, 4\,.
\end{equation}

\subsubsection{Estimation of $\mathbf{\Omega_1}$, $\mathbf{\Omega_2}$, $\mathbf{\Omega_3}$ and $\mathbf{\Omega_4}$}
\indent

First we consider the sums $\Omega_1$ and $\Omega_2$. 
Using  \eqref{U1}, \eqref{U2}, \eqref{Omegai} and arguing as in \cite{Petrov-Tolev}, we obtain
\begin{equation}\label{Omega1est1}
\Omega_1\ll \frac{N^{2\gamma-1}}{\log N}
\end{equation}
and
\begin{equation}\label{Omega2est1}
\Omega_2\ll \frac{N^{2\gamma-1}}{\log N}\,.
\end{equation}
Next we consider the sum $\Omega_4$. From \eqref{U4}, we have
\begin{equation}\label{U4U5}
U_4\ll|U_5|\log N,
\end{equation}
where
\begin{equation}\label{U5}
U_5=\sum_{L<l\le 2L}b(l)\sum_{M<m\le 2M\atop{P/l<m\leq 2P/l}}a(m)e\big(f(m, l)\big)
\end{equation}
and where
\begin{equation}\label{ParU5}
a(m)\ll N^\varepsilon\,, \quad b(l)\ll N^\varepsilon\,, \quad P^{1/3}\ll M\ll P^{1/2}\ll L\ll P^{2/3}\,, \quad  ML\asymp P\,.
\end{equation}
Using \eqref{U4U5}, \eqref{U5}, \eqref{ParU5} and Cauchy's inequality, we find
\begin{equation}\label{U4est1}
|U_4|^2\ll N^\varepsilon L\sum_{L<l\le 2L}\bigg|\sum_{M_1<m\le M_2}a(m)e\big(f(m, l)\big)\bigg|^2,
\end{equation}
where
\begin{equation}\label{maxmin1}
M_1=\max{\bigg\{M,\frac{P}{l}\bigg\}},\quad
M_2=\min{\bigg\{2M, \frac{2P}{l}\bigg\}}\,.
\end{equation}
Now \eqref{ParU5} -- \eqref{maxmin1}  and Lemma \ref{Iwaniec-Kowalski} with $Q$ such that
\begin{equation}\label{QM}
1\leq Q\leq M
\end{equation}
give us
\begin{align}\label{U4est2}
|U_4|^2&\ll N^\varepsilon L  \sum_{L<l\le 2L}\frac{M}{Q}
\sum_{|q|\leq Q}\bigg(1-\frac{|q|}{Q}\bigg)
\sum_{M_1<m\le M_2\atop{M_1<m+q\le M_2}}a(m+q)\overline{a(m)}e\big(f(m+q,l)-f(m,l)\big)\nonumber\\
&\ll N^\varepsilon\Bigg(\frac{(LM)^2}{Q}+\frac{LM}{Q}\sum_{0<|q|\leq Q}
\sum_{M<m\le 2M\atop{M<m+q\le 2M}}\bigg|\sum_{L_1<l\le L_2}e\big(g(l)\big)\bigg|\Bigg)\nonumber\\
&\ll N^\varepsilon\Bigg(\frac{(LM)^2}{Q}+\frac{LM}{Q}\sum_{1\leq q\leq Q}
\sum_{M<m\le 2M-q}\bigg|\sum_{L_1<l\le L_2}e\big(g(l)\big)\bigg|\Bigg)\,,
\end{align}
where
\begin{equation}\label{maxmin2}
L_1=\max{\bigg\{L, \frac{P}{m}, \frac{P}{m+q}\bigg\}},\quad
L_2=\min{\bigg\{2L,\frac{2P}{m},\frac{2P}{m+q}\bigg\}}
\end{equation}
and
\begin{equation}\label{gl}
g(l)=f(m+q,l)-f(m,l)\,.
\end{equation}
Consider the function $g(l)$.
From \eqref{htT}, \eqref{fml} and \eqref{gl}, it follows
\begin{equation*}
g(l)=\int\limits_{m}^{m+q}f_t^\prime(t,l)\,dt=\int\limits_{m}^{m+q}lh^\prime(tl)\,dt.
\end{equation*}
Hence
\begin{equation}\label{gkl}
g^{(k)}(l)=\int\limits_{m}^{m+q}\Big(kt^{k-1}h^{(k)}(tl)+lt^{k}h^{(k+1)}(tl)\Big)dt\,, \quad k\geq1\,.
\end{equation}
Taking into account \eqref{htT} and \eqref{gkl}, we obtain
\begin{align}
\label{gklest1}
&g^{\prime}(l)=\int\limits_{m}^{m+q}\Big(rc^2(tl)^{c-1}+v(tl)^{c-1}\big(T-(tl)^c\big)^{\gamma-2}\big((tl)^c-cT\big)\Big)dt\,, \\
\label{gklest2}
&g^{(k)}(l)=\int\limits_{m}^{m+q} \Big(\Psi_1(t, l)+\Psi_2(t, l)\Big)\,dt\,, \quad k\geq2\,,
\end{align}
where
\begin{align}
\label{Psi1}
&\Psi_1(t, l)=r c^2(c-1)(c-2) \cdots (c-k+1)t^{c-1}l^{c-k}\,,\\
\label{Psi2}
&\Psi_2(t, l)=v(c-1)Tt^{c-1}l^{c-k}\big(T-(tl)^c\big)^{\gamma-k-1}\sum_{i=1}^k\phi_i(c) T^{k-i}(tl)^{(i-1)c}\,, 
\end{align}
where $\phi_i(x)\in \mathbb{Z}[x]$ are polynomials of degree $k-1$ for $k\geq 2$.
If $t\in[m, m+q]$, then
\begin{equation}\label{tlP}
tl\asymp P\,.
\end{equation}
By \eqref{P}, \eqref{Psi1}, \eqref{Psi2}, \eqref{tlP} and the condition $N\leq T\leq2N$, we get
\begin{equation}\label{Psi1Psi2est}
|\Psi_1(t, l)|\asymp |r| m^{k-1} N^{1-k\gamma}\,, \quad |\Psi_2(t, l)|\asymp v m^{k-1} N^{(1-k)\gamma}
\end{equation}
uniformly for $t\in[m, m+q]$.
On the one hand from  \eqref{gklest1},  \eqref{gklest2}  and \eqref{Psi1Psi2est} we conclude that there exists sufficiently small constant $\alpha_1>0$ such that if
$|r|\leq \alpha_1vN^{\gamma-1}$, then $|g^{(k)}(l)|\asymp qv m^{k-1} N^{(1-k)\gamma}$.
On the other hand from \eqref{gklest1},  \eqref{gklest2}  and \eqref{Psi1Psi2est}  it follows that there exists sufficiently large constant $A_1>0$ such that if
$|r|\geq A_1vN^{\gamma-1}$, then $|g^{(k)}(l)|\asymp q|r| m^{k-1} N^{1-k\gamma}$.
For our purpose, it is convenient to decompose the sum $\Omega_4$ into four parts, i.e. from \eqref{Omegai}, we write
\begin{equation}\label{Omega4decomp}
\Omega_4=\Omega^{(1)}_4+\Omega^{(2)}_4+\Omega^{(3)}_4+\Omega^{(4)}_4 \,,
\end{equation}
where
\begin{align}
\label{Omega41}
&\Omega^{(1)}_4=\sum_{d\leq D}\sum\limits_{h\leq H}\frac{1}{h}\sum_{A_1vN^{\gamma-1}\leq|r|\leq R}\sup_{T\in[N, N+2]}|U_4|\,,\\
\label{Omega42}
&\Omega^{(2)}_4=\sum_{d\leq D}\sum\limits_{h\leq H}\frac{1}{h}\sum_{\alpha_1vN^{\gamma-1}<r< A_1vN^{\gamma-1}}\sup_{T\in[N, N+2]}|U_4|\,,\\
\label{Omega43}
&\Omega^{(3)}_4=\sum_{d\leq D}\sum\limits_{h\leq H}\frac{1}{h}\sum_{|r|\leq \alpha_1vN^{\gamma-1}}\sup_{T\in[N, N+2]}|U_4|\,,\\
\label{Omega44}
&\Omega^{(4)}_4=\sum_{d\leq D}\sum\limits_{h\leq H}\frac{1}{h}\sum_{-A_1vN^{\gamma-1}<r< -\alpha_1vN^{\gamma-1}}\sup_{T\in[N, N+2]}|U_4|\,.
\end{align}

\textbf{Upper bound for $\mathbf{\Omega^{(1)}_4}$}

We recall that the constant $A_1$ is chosen in such a way, that if $|r|\geq A_1vN^{\gamma-1}$,
then uniformly for $l\in(L_1, L_2]$ we have 
\begin{equation}\label{gklasymp}
|g^{(k)}(l)|\asymp q|r| m^{k-1} N^{1-k\gamma}\asymp q|r| N^{1-\gamma} \big(Pm^{-1}\big)^{1-k}\,, \quad k\geq1\,.
\end{equation}
According to Lemma \ref{ExppairBourgain}, the pair
\begin{equation*}
BA\left(\frac{13}{84}+\varepsilon, \frac{55}{84}+\varepsilon\right)=\left(\frac{55}{194}+\varepsilon,\, \frac{110}{194}+\varepsilon\right)
\end{equation*}
is an exponent pair.
Using \eqref{P}, \eqref{maxmin2}, \eqref{gklasymp} and Lemma \ref{Exponentpairs} with exponent pair
\begin{equation*}
\left(\frac{55}{194}+\varepsilon,\,\frac{110}{194}+\varepsilon\right)
\end{equation*}
we derive
\begin{align}\label{sumL1L2}
\sum_{L_1<l\le L_2}e(g(l))&\ll \big( q|r| N^{1-\gamma}\big)^{\frac{55}{194}+\varepsilon}\big(Pm^{-1}\big)^{\frac{110}{194}+\varepsilon}+\big( q|r| N^{1-\gamma}\big)^{-1}\nonumber\\
&\ll q^{\frac{55}{194}+\varepsilon} |r|^{\frac{55}{194}+\varepsilon} m^{-\frac{55}{97}+\varepsilon} N^{\frac{55}{194}+\frac{55\gamma}{194}+\varepsilon}\,.
\end{align}
Now \eqref{ParU5}, \eqref{U4est2} and \eqref{sumL1L2} yield 
\begin{align*}
|U_4|^2&\ll N^\varepsilon\Bigg(\frac{(LM)^2}{Q}+\frac{LM}{Q}\sum_{1\leq q\leq Q}
\sum_{M<m\le 2M-q}q^\frac{55}{194} |r|^\frac{55}{194} m^{-\frac{55}{97}} N^{\frac{55}{194}+\frac{55\gamma}{194}}\Bigg)\nonumber\\
&\ll N^\varepsilon\bigg(N^{2\gamma}Q^{-1}+N^\gamma Q^{-1}\sum_{1\leq q\leq Q}
q^\frac{55}{194} |r|^\frac{55}{194} M^\frac{42}{97} N^{\frac{55}{194}+\frac{55\gamma}{194}}\bigg)\nonumber\\
&\ll N^\varepsilon\Big(N^{2\gamma}Q^{-1}+Q^\frac{55}{194} |r|^\frac{55}{194} N^{\frac{55}{194}+\frac{291\gamma}{194}}\Big)
\end{align*}
which together with \eqref{z}, \eqref{RdN} and  \eqref{Omega41} implies
\begin{align}\label{Omega41est1}
\Omega^{(1)}_4&\ll N^\varepsilon\sum_{d\leq D}\sum\limits_{h\leq H}\frac{1}{h}\sum_{|r|\leq d^2N^{1-\gamma}\log^8N}
\Big(N^\gamma Q^{-\frac{1}{2}}+Q^\frac{55}{388} |r|^\frac{55}{388} N^{\frac{55}{388}+\frac{291\gamma}{388}}\Big)\nonumber\\
&\ll N^\varepsilon\Big(D^3NQ^{-\frac{1}{2}}+Q^\frac{55}{388} D^\frac{637}{194} N^{\frac{249}{194}-\frac{38\gamma}{97}}\Big)\nonumber\\
&\ll N^\varepsilon\Big(NQ^{-\frac{1}{2}}+Q^\frac{55}{388} N^{\frac{249}{194}-\frac{38\gamma}{97}}\Big)\,.
\end{align}
Choosing  
\begin{equation}\label{Q}
Q= \Big[N^{\frac{152\gamma}{249}-\frac{110}{249}}\Big]
\end{equation}
we see that the condition \eqref{QM} holds. Bearing in mind \eqref{Omega41est1} and \eqref{Q}, we deduce
\begin{equation}\label{Omega41est2}
\Omega^{(1)}_4\ll \frac{N^{2\gamma-1}}{\log N}\,.
\end{equation}

\textbf{Upper bound for $\mathbf{\Omega^{(2)}_4}$}

From (\cite{Petrov-Tolev}, p. 55), \eqref{P}, \eqref{ParU5}, \eqref{U4est2} and \eqref{maxmin2}  we have
\begin{equation*}
U_4 \ll N^\varepsilon\Big(N^\gamma Q^{-\frac{1}{2}}+v^\frac{1}{4} Q^\frac{1}{4} N^\frac{7\gamma}{8}+v^{-\frac{1}{4}} Q^{-\frac{1}{4}} N^\frac{7\gamma}{8}
+v^\frac{1}{12} Q^\frac{1}{12} N^\frac{11\gamma}{12}+v^{-\frac{1}{12}} Q^{-\frac{1}{12}} N^\frac{23\gamma}{24}\Big)
\end{equation*}
which together with \eqref{z}, \eqref{HdN}, \eqref{vdh}, \eqref{RdN} and \eqref{Omega42} leads to 
\begin{align}\label{Omega42est1}
\Omega^{(2)}_4&\ll N^\varepsilon\sum_{d\leq D}\sum\limits_{h\leq H}\frac{1}{h}\sum_{r<A_1\log N}
\Bigg(N^\gamma Q^{-\frac{1}{2}}+\left(\frac{h}{d^2}\right)^\frac{1}{4} Q^\frac{1}{4} N^\frac{7\gamma}{8}+\left(\frac{h}{d^2}\right)^{-\frac{1}{4}} Q^{-\frac{1}{4}} N^\frac{7\gamma}{8}\nonumber\\
&\hspace{56mm}+\left(\frac{h}{d^2}\right)^\frac{1}{12} Q^\frac{1}{12} N^\frac{11\gamma}{12}+\left(\frac{h}{d^2}\right)^{-\frac{1}{12}} Q^{-\frac{1}{12}} N^\frac{23\gamma}{24}\Bigg)\nonumber\\
&\ll N^\varepsilon \Big(N^\gamma Q^{-\frac{1}{2}}+Q^\frac{1}{4} N^{\frac{1}{4}+\frac{5\gamma}{8}}+Q^{-\frac{1}{4}} N^\frac{7\gamma}{8}
+Q^\frac{1}{12} N^{\frac{1}{12}+\frac{5\gamma}{12}}+Q^{-\frac{1}{12}} N^\frac{23\gamma}{24}\Big)\,.
\end{align}
By \eqref{Q} and \eqref{Omega42est1} it follows
\begin{equation}\label{Omega42est2}
\Omega^{(2)}_4\ll \frac{N^{2\gamma-1}}{\log N}\,.
\end{equation}

\textbf{Upper bound for $\mathbf{\Omega^{(3)}_4}$ and $\mathbf{\Omega^{(4)}_4}$ }

Arguing as in (\cite{Petrov-Tolev}, p. 55) for the sums defined by \eqref{Omega43} and \eqref{Omega44}, we obtain
\begin{equation}\label{Omega43est1}
\Omega^{(3)}_4\ll \frac{N^{2\gamma-1}}{\log N}
\end{equation}
and
\begin{equation}\label{Omega44est1}
\Omega^{(4)}_4\ll \frac{N^{2\gamma-1}}{\log N}\,.
\end{equation}
Taking into account \eqref{Omega4decomp}, \eqref{Omega41est2}, \eqref{Omega42est2} , \eqref{Omega43est1} and \eqref{Omega44est1}, we establish 
\begin{equation}\label{Omega4est1}
\Omega_4\ll \frac{N^{2\gamma-1}}{\log N}\,.
\end{equation}
Working as for $\Omega_4$ we find
\begin{equation}\label{Omega3est1}
\Omega_3\ll \frac{N^{2\gamma-1}}{\log N}\,.
\end{equation}

\subsubsection{Upper bound for $\mathbf{\Sigma_0}$ and $\mathbf{\Sigma_1}$}
\indent

Summarizing  \eqref{Sigmajest3}, \eqref{Omega1est1}, \eqref{Omega2est1}, \eqref{Omega4est1}  and \eqref{Omega3est1}, we get
\begin{equation}\label{Sigmajest4}
|\Sigma_j|\ll \frac{N^{2\gamma-1}}{\log N}\,, \quad j=0, 1\,,
\end{equation}

\subsection{ Lower bound for $\mathbf{\Gamma_1}$}
\indent

Bearing in mind \eqref{Gamma1est1}, \eqref{Gamma3est} and \eqref{Sigmajest4}, we deduce
\begin{equation}\label{Gamma1est2}
\Gamma_1\gg N^{2\gamma-1}\,.
\end{equation}

\section{Estimation of $\mathbf{\Gamma_2}$}\label{SectionGamma2}
\indent

In this section we need a lemma that gives us information about the upper bound
of the number of solutions of the binary equation \eqref{binary2} with an additional condition imposed for $m$.
\begin{lemma}\label{Thenumberofsolutions}
Let $1<c<2$  and $N$ be a positive integer.
Then for the number of solutions $B(N)$ of the diophantine equation
\begin{equation*}
[p^c]+[m^c]=N
\end{equation*}
with prime $p$ and positive integer $m$ such that $m\equiv 0\,(d)$, we have
\begin{equation}\label{BNupperbound}
B(N)\ll\frac{N^{\frac{2}{c}-1}}{d\log N}\,.
\end{equation}
\end{lemma}
\begin{proof}
We have
\begin{equation}\label{BN}
B(N)=\sum\limits_{[p^c]+[m^c]=N\atop{m\equiv 0\,(d)}}1=\int\limits_{0}^{1} S_1(t)S_2(t) e(-Nt)\,dt\,,
\end{equation}
where
\begin{align}
\label{S1}
&S_1(t)=\sum\limits_{p\leq N^{1/c}} e(t [p^c])\,,\\
\label{S2}
&S_2(t)=\sum\limits_{m\leq N^{1/c}\atop{m\equiv 0\,(d)}} e(t [m^c])\,.
\end{align}
From \eqref{BN} -- \eqref{S2}, partial integration and the trivial estimations
\begin{equation}\label{S12trivest}
S_1(t)\ll \frac{N^{\frac{1}{c}}}{\log N} \,, \quad S_2(t)\ll \frac{N^{\frac{1}{c}}}{d}
\end{equation}
we deduce
\begin{align}\label{BNest}
B(N)&=-\frac{1}{2\pi i}\int\limits_{0}^{1}\frac{S_1(t)S_2(t)}{N}\,d\,e(-Nt)\nonumber\\
&=-\frac{S_1(t) S_2(t)e(-Nt)}{2\pi iN}\Bigg|_{0}^{1}
+\frac{1}{2\pi iN}\int\limits_{0}^{1}e(-Nt)\,d\Big(S_1(t)S_2(t)\Big)\nonumber\\
&\ll \frac{N^{\frac{2}{c}-1}}{d\log N}+N^{-1}|\Omega|\,,
\end{align}
where
\begin{equation}\label{Omega}
\Omega=\int\limits_{0}^{1}e(-Nt)\,d\Big(S_1(t)S_2(t)\Big)\,.
\end{equation}
We consider $\Omega$. Put
\begin{equation*}
f(t)=S_1(t)S_2(t)\,,\quad 0\leq t\leq1\,.
\end{equation*}
Since the set of zeros of Re$f'(t)$ is finite on $[0, 1]$, then the integral \eqref{Omega} can be divided into a finite number of integrals, 
such that the function $f(t)$ is univalent on the corresponding interval of integration. 
Assume that Re$f'(t)$ has $k$ zeros in $[0, 1]$. When Re$f'(t)$ has no zeros in $[0, 1]$ we work analogously. Let 
\begin{equation*}
\textmd{Re}f'(t_1)=\textmd{Re}f'(t_2)=\cdots=\textmd{Re}f'(t_k)=0\,, \quad \hbox{ where } \quad 0\leq t_1<t_2<\cdots<t_k\leq1
\end{equation*}
and
\begin{equation}\label{gammai}
\gamma_i \, :\, z=f(t)=S_1(t)S_2(t)\,,\quad \textmd{Re}f'(t)\neq0\,,\quad  t_i<t<t_{i+1}\,, \quad 0\leq i \leq k\,,
\end{equation}
where $t_0=0$, $t_{k+1}=1$. Using \eqref{Omega} and \eqref{gammai}, we obtain
\begin{equation}\label{Omegaest1}
\Omega=\sum\limits_{i=0}^k \int\limits_{\gamma_i} e\Big(-Nf^{-1}(z)\Big)\,dz\,.
\end{equation}
Taking into account \eqref{S12trivest}, \eqref{gammai} and that the integrals in \eqref{Omegaest1} are independent of path, we derive
\begin{equation}\label{Omegaest2}
\Omega=\sum\limits_{i=0}^k\int\limits_{\overline{\gamma}_i}e\Big(-Nf^{-1}(z)\Big)\,dz
\ll\sum\limits_{i=0}^k\int\limits_{\overline{\gamma}_i} |dz|
\ll \sum\limits_{i=0}^k\big(|f(t_i)|+|f(t_{i+1})|\big) \ll \frac{N^{\frac{2}{c}}}{d\log N}\,,
\end{equation}
where $\overline{\gamma}_i$ is the line segment connecting the points $f(t_i)$ and $f(t_{i+1})$.
Bearing in mind \eqref{BNest} and \eqref{Omegaest2}, we establish the upper bound \eqref{BNupperbound}.
\end{proof}
We are now in a good position to estimate the sum $\Gamma_2$.
From \eqref{Gamma2} and Lemma \ref{Thenumberofsolutions}, we get
\begin{equation}\label{Gamma2est}
\Gamma_2\ll\sum_{d>D}\frac{N^{2\gamma-1}}{d^2}\ll N^{2\gamma-1}D^{-1}\,.
\end{equation}

\section{The end of the proof}\label{Sectionfinal}
\indent

Summarizing \eqref{z}, \eqref{Gammadecopm}, \eqref{Gamma1est2} and \eqref{Gamma2est}, we establish the lower bound
\begin{equation*}
\Gamma\gg N^{2\gamma-1}\,.
\end{equation*}
This completes the proof of  Theorem \ref{Theorem}.

\vskip20pt
\footnotesize
\begin{flushleft}
S. I. Dimitrov\\
\quad\\
Faculty of Applied Mathematics and Informatics\\
Technical University of Sofia \\
Blvd. St.Kliment Ohridski 8 \\
Sofia 1000, Bulgaria\\
e-mail: sdimitrov@tu-sofia.bg\\
\end{flushleft}

\begin{flushleft}
Department of Bioinformatics and Mathematical Modelling\\
Institute of Biophysics and Biomedical Engineering\\
Bulgarian Academy of Sciences\\
Acad. G. Bonchev Str. Bl. 105, Sofia 1113, Bulgaria \\
e-mail: xyzstoyan@gmail.com\\
\end{flushleft}

\end{document}